\documentclass[12pt]{amsart}
\usepackage{amsmath, amsthm, amscd, amsfonts}

\newcommand{\MM}[0]{
{\bf M} }
\newcommand{\NN}[0]{
{\mathbb{N}}
}
\newcommand{\norm}[2]{
\left\| #2 \right\|_{#1} }

\def\bes{\begin{eqnarray*}}

\def\ens{\end{eqnarray*}}

\theoremstyle{definition}

\theoremstyle{remark}

\numberwithin{equation}{section}
\newtheorem{thm}{Theorem}[section]
 \newtheorem{cor}[thm]{Corollary}
 \newtheorem{lem}[thm]{Lemma}
 \newtheorem{prop}[thm]{Proposition}
 \theoremstyle{definition}
 \newtheorem{defn}[thm]{Definition}
 \theoremstyle{remark}

 \numberwithin{equation}{section}
\theoremstyle{plain}

\def\bes{\begin{eqnarray*}}
\def\ens{\end{eqnarray*}}
\theoremstyle{definition}
\theoremstyle{remark}
\numberwithin{equation}{section}

\begin{document}

\title[Multipliers for  $p$-Bessel sequences ]{Multipliers for  $p$-Bessel sequences in Banach spaces}
\author[A. Rahimi \& P. Balazs]{A. Rahimi$^*$ \and P. Balazs$^\dagger$}
\address{$^*$ Department of Mathematics, University of Maragheh, P. O. Box 55181-83111, Maragheh, Iran.}
\email{asgharrahimi@yahoo.com}
\address{$^\dagger$ Acoustics Research Institute, Austrian Academy of Sciences, Wohllebengasse 12-14, 1040 Wien, Austria.}
\email{Peter.Balazs@oeaw.ac.at}

\subjclass[2000]{Primary 42C40; Secondary 41A58, 47A58,.}

\keywords{Multiplier operator, Bessel sequence, Frame, Schauder
basis, $p-$frame, $(p,q)-$Bessel multiplier, $p-$Schatten
operator, Nuclear operator, $(r,p,q)-$Nuclear operators.}

\begin{abstract}
Multipliers have been recently introduced as operators for Bessel
sequences and frames in Hilbert spaces. These operators
 are defined by a fixed multiplication pattern (the symbol)
which is inserted between the analysis and synthesis operators. In
this paper, we will generalize the concept of Bessel multipliers for
$p$-Bessel and $p$-Riesz sequences in Banach spaces. It will be
shown that bounded symbols lead to bounded operators. Symbols
converging to zero induce compact operators. Furthermore, we will
give sufficient conditions for multipliers to be nuclear operators.
Finally, we will show the continuous dependency of the multipliers
on their parameters.

\end{abstract} \maketitle
\normalsize
\section{Introduction and Preliminaries}

\noindent
\subsection{Introduction}
In \cite{schatt1}, R. Schatten provided a detailed study of ideals
of compact operators using their singular decomposition. He
investigated the operators of the form
$\sum_{k}\lambda_{k}\varphi_{k}\otimes\overline{\psi_{k}}$ where
$(\phi_{k})$ and $(\psi_{k})$ are orthonormal families. In
\cite{xxlmult1} the orthonormal families were replaced with Bessel
and frame sequences to define Bessel and frame multipliers.
\begin{defn}
Let $\mathcal{H}_{1}$ and $\mathcal{H}_{2}$ be Hilbert spaces, let
$(\psi_{k})\subseteq\mathcal{H}_{1}$ and
$(\phi_{k})\subseteq\mathcal{H}_{2}$ be Bessel sequences. Fix
$m=(m_k)\in l^{\infty}(K)$. The operator ${\bf M}_{m, ( \phi_k),
(\psi_k)} : \mathcal{H}_{1} \rightarrow \mathcal{H}_{2}$ defined by
$$ {\bf M}_{m, (\phi_k), ( \psi_k )} (f)  =  \sum \limits_k m_k
\langle f,\psi_k\rangle \phi_k $$ is called the Bessel multiplier
for the Bessel sequences $(\psi_{k})$ and $(\phi_{k})$. The sequence
$m$ is called the symbol of {\bf M}.
\end{defn}
Several basic properties of these operators were investigated in
\cite{xxlmult1}. For a theoretical approach it is very natural to
extend this notion and consider such operators in more general
settings. For $p$-Bessel sequences in Banach spaces is leading to
interesting results in functional analysis and operator theory.

We are going to show only theoretical properties. Nevertheless it
should be mentioned, that multipliers are not only interesting from
a theoretical point of view, see e.g.
\cite{xxlframehs07,xxlframoper1,doetor09}, but they are also used in
applications, in particular in the fields of audio and acoustics.
The first frame multipliers investigated were Gabor (frame)
multipliers \cite{feinow1}.
In signal processing they are used 
 under the name
'Gabor filters' \cite{hlawatgabfilt1} as a particular choice to
implement a time-variant filter. In computational auditory scene
analysis they are known by the name 'time-frequency masks'
\cite{wanbro06} and are used to extract single sound source out of a
mixture of sounds in a way linked to human auditory perception.
 In real-time implementations of filtering systems, they approximate time-invariant filters \cite{stx05} as they are easily implementable.
 On the other hand, as a particular way to implement time-variant filters, they are used for example for
 sound morphing \cite{DepKronTor07} or psychoacoustical
  modeling \cite{xxllabmask1}.
  In general the idea for a Gabor (or wavelet) multiplier is to amplify or attenuate parts of audio signal, which can be separated in the time-frequency plane.

Clearly Banach space theory is right at the foundation of functional
analysis and operator theory, and as such is relevant for theory.
But it recently also has become more and more important for
time-frequency analysis, see e.g. \cite{grhe99}. It is used in
engineering applications in compressed sensing, refer e.g. to
\cite{rascva06}, as well as in audio or image sampling
\cite{aldrgroech1}. Applications in wireless communication can be
envisioned \cite{grrz08,grhahlmasc07}. Therefore we hope that the
results in this paper are not only interesting from a theoretical
point of view but can be applied in the not-too-far future, for
example by combining multipliers with the concept of sparsity and
persistence \cite{Kowalski10sparsity}.
\\

 In this paper, we define and
investigate multipliers in Banach spaces. In Section
\ref{sec:prel0}, we will give the basic definitions and known
results needed. In Section \ref{sec:multban0} we will give basic
results for multipliers for $p$-Bessel sequences.  In particular we
will show that multipliers with bounded symbols are well-defined
bounded operators with unconditional convergence and that symbols
converging to zero correspond to compact operators.
 Section \ref{sec:nucl0} will look at sufficient
conditions for multipliers to be $(r,p,q)$-nuclear. Finally, in
Section \ref{sec:changing0}, we will look at how the multipliers
depend on the given parameters, i.e. the analysis and synthesis
sequences as well as the symbol. We will show that this dependence
is continuous, using a similarity of sequences in an $l^p$ sense.
\subsection{Preliminaries} \label{sec:prel0}
We will only consider reflexive Banach spaces. We will assume that
$p, q>0$  are real numbers such that $\frac{1}{p}+\frac{1}{q} = 1$.
For any separable Banach space, we can define $p$-frame and
$p$-Bessel sequences \cite{alsuta01,chst03} as follows:

\begin{defn}
A countable family $(g_{i})_{i\in I}\subseteq X^{*}$ is a $p$-frame
for the Banach space $X $ $(1<p<\infty)$ if constants $A, B>0$ exist
such that
$$ A\|f\|_{X}\leq \left(\sum_{i\in I} |g_{i}(f)|^{p}\right)^{\frac{1}{p}}\leq
B\|f\|_{X}\quad\textrm{for all}\quad
 f \in X.$$
 It is called a $p$-Bessel sequence with bound $B$ if the second inequality holds.
\end{defn}

For $p$-Bessel sequences we can define the analysis operator $ U : X
\rightarrow l^p$ with $U (f) = \left( g_i (f) \right)$. Following
the definition we see that $ \| U \| \le B$. Furthermore, let $T :
l^q  \rightarrow X^*$ be the synthesis operator defined by $T \left(
\left(d_i  \right) \right) = \sum \limits_i d_i g_i$.
\begin{prop} \label{sec:besselbound1} \cite{chst03}   $( g_{i} ) \subseteq X^*$ is a $p$-Bessel sequence with bound $B$
if and only if $T$ is a well-defined (hence bounded) operator from
$l^q$ into $X^*$ and $\| T \| \le B$. In this case $T \left(
\left(d_i  \right) \right) = \sum \limits_i d_i g_i$ converges
unconditionally.
\end{prop}
Furthermore $\| g_i \| \le B$.

\begin{defn}
Let $Y$ be a Banach space. A family $(g_{i})_{i\in I}\subset Y$ is a
$q$-Riesz sequence $(1<q<\infty)$ for $Y$ if
 constants
$A, B>0$ exist such that for all finite scalar sequence $(d_{i})$,
\begin{equation}\label{qRiesz} A(\sum_{i\in I}|d_{i}|^{q})^{\frac{1}{q}}\leq
\parallel\sum_{i\in I} d_{i}g_{i}\parallel_{Y}\leq B(\sum_{i\in
I}|d_{i}|^{q})^{\frac{1}{q}}.\end{equation}

The family is called a $q$-Riesz basis $(1<q<\infty)$ for $Y$ if it
fulfills (\ref{qRiesz}) and $\overline{span}\{g_{i}\}_{i\in I}=Y$.
\end{defn}
\par It
immediately follows from the definition that, if $(g_{i})_{i\in I}$
is a $q$-Riesz basis then $A\leq \|g_{i}\|_{Y}\leq B$ for all $i\in
I.$
 Any $q$-Riesz basis for $X^*$ is a $p$-frame for $X$
\cite{chst03}. The following proposition shows a connection between
these two notions similar to the case for Hilbert spaces:
\begin{prop}\label{ri}\cite{chst03}
Let $(g_{i})_{i\in I}\subset X^*$ be a $p$-frame for $X$. Then the
following are equivalent:
\begin{enumerate}
\item $(g_{i})_{i\in I}$ is a $q$-Riesz basis for $X^*$.
\item  If  $(d_{i})_{i\in I} \in l^q$ and $\sum d_{i}g_{i}=0$, then $d_{i}=0$ for all $i\in I$.
\item $(g_{i})_{i\in I}$ has a biorthogonal sequence $(f_{i})_{i\in I}\subset
X$, i.e., a family for which $g_{i}(f_{j})=\delta_{i,j}$ (Kronecker
delta), for all $i,j\in I$.
\end{enumerate}
\end{prop}
\begin{thm} \label{sec:dualriesz1}\cite{chst03}
Let $(g_{i})_{i\in I}\subset X^*$ be a $q$-Riesz basis for $X^*$
with bounds $A,B$. Then there exists a unique $p$-Riesz basis
$(f_{i})_{i\in I}\subset X$ for which
$$ f = \sum \limits_i g_i (f) f_i\quad\quad \mbox{ and }\quad\quad g = \sum \limits_i f_i(g) g_i$$
for all $f \in X $ and $ g \in X^*$. The bounds of $(f_{i})_{i\in
I}$ are  $1/B$ and $1/A$, and it is biorthogonal to $( g_i)$.
\end{thm}

\begin{defn} We will call the unique sequence of Theorem \ref{sec:dualriesz1} {\em the dual} of $(g_i)$ and denote it by $( \tilde g_i)$.
\end{defn}
\subsection{Perturbation of $p$-Bessel sequences}
Similar to the case for Hilbert spaces  perturbation results for
Banach spaces are possible.
\begin{thm}  \label{sec:perturboper1} \cite{cach97}
Let $U : X \rightarrow Y$ be a bounded operator, $X_0$ a dense
subspace of $X$ and $V: X \rightarrow Y$ a linear mapping. If for
$\lambda_1, \mu > 0$ and $0 \le \lambda_2 < 1$
$$ \left\| U x - V x \right\| \le \lambda_1 \left\|U x \right\| + \lambda_2 \left\| V x \right\| + \mu \left\|x\right\|, $$
for all $x \in X_0$, then $V$ is a bounded linear operator.
\end{thm}

\begin{cor} \label{sec:framesimi1} Let $(\psi_k) \subseteq X^*$ be a $p$-Bessel sequence. 
\begin{enumerate}
\item  If $(\phi_k) \subseteq X^*$ is a sequence with $\left( \sum \limits_k \norm{X^*}{\psi_k - \phi_k}^p\right)^{1/p} < \mu <\infty$,
 then $(\phi_k)$ is a $p$-Bessel sequences with bound $B + \mu$. 
\item
Let $(\phi_k^{(l)})$ be a sequence such that for all $\varepsilon$
there exists an $N_\varepsilon$ with $\left( \sum \limits_k
\norm{X^*}{\psi_k - \phi_k^{(l)}}^p\right)^{1/p} < \varepsilon$ for
all $l \ge N_\epsilon$. Then the sequence $(\phi_k^{(l)})$ is a
Bessel sequence and for all $l \ge N_{\epsilon}$
$$ \norm{Op}{U_{\phi_k^{(l)}} - U_{(\psi_k)}}
< \varepsilon \quad and\quad  \norm{Op}{T_{\phi_k^{(l)}} -
T_{(\psi_k)}} < \varepsilon .$$
\end{enumerate}
\end{cor}
\begin{proof} For any $c\in \ell^{p} $ with finite support, we have
\begin{eqnarray*} \left\| T_{\Psi} c - T_{\Phi} c \right\|&=& \left\| \sum \limits_i c_i \left( \psi_i - \phi_i \right)
\right\|\\ &\le&  \sum \limits_i \left| c_i \right| \left\| \psi_i -
\phi_i \right\|
\\ &\le& \left(\sum \limits_i \left| c_i \right|^q  \right)^{1/q}  \cdot \left( \sum \limits_i \left\| \psi_i - \phi_i \right\|^p\right)^{1/p} \\ &\le& \left\| c \right\|_q \mu .
\end{eqnarray*}
Furthermore
\begin{eqnarray*} \left\| U_{\Psi} f - U_{\Phi} f \right\| &=& \left\|( \psi_i(f) - \phi_i(f)) \right\|_p \\ &=&
 \left( \sum \limits_i \left| \psi_i(f) - \phi_i(f) \right|^p\right)^{1/p} \\ &\le&
\left( \sum \limits_i \left\| \psi_i - \phi_i
\right\|_{X^*}^p\right)^{1/p} \norm{X}{f}
\\ &\le& \mu  \norm{X}{f}.
\end{eqnarray*}
We can apply Theorem \ref{sec:perturboper1} for $\lambda_1 =
\lambda_2 = 0$ and use Proposition \ref{sec:besselbound1}.

Part (2) can be proved in an analogue way.
\end{proof}
For a full treatment of perturbation for frames and Bessel sequences
in Banach spaces refer to \cite{stoeva09}.

\begin{defn} Let $(\psi_k)_{k \in K} \subseteq X^*$ and $(\psi_k^{(l)})_{k \in K} \subseteq X^*$ be a sequence of elements for all $l \in \NN$.
 The sequences $(\psi_k^{(l)})$ are said to \em converge to $(\psi_k)$ in an $l^p$-sense\em , denoted
  by $(\psi_k^{(l)}) \stackrel{l^p}{\longrightarrow} (\psi_k)$, if for any $\varepsilon > 0$ there exists $N_\varepsilon > 0$
  such that  $\left(\sum \limits_k \|\psi_k^{(l)} - \psi_k\|_{X^*}^p\right)^{\frac{1}{p}} < \varepsilon$, for all
   $\ l \ge N_\varepsilon$.
\end{defn}
This is related to the notions of `quadratic closeness'
\cite{young1}, and  `Bessel norm'\cite{xxlmult1}.

\section{Multipliers for $p$-Bessel sequences}\label{sec:multban0}
\begin{lem} \label{sec:multopinf1}
Let $(\psi_k) \subseteq X_1^*$ be a  $p$-Bessel sequence for $X_1$
with bound $B_1$, let $(\phi_k) \subseteq X_2$ be a $q$-Bessel
sequence for $X_2^*$ with bound $B_2$, let $m \in l^\infty$.
The operator ${\bf M}_{m, ( \phi_k), (\psi_k)} : X_1 \rightarrow
X_2$ defined by $$ {\bf M}_{m, (\phi_k), ( \psi_k )} (f)  =  \sum
\limits_k m_k \psi_k (f)\phi_k .$$ is well defined. This sum
converges unconditionally and
$$\norm{Op}{\bf M} \le  B_2 B_1 \cdot
\norm{\infty}{m} .$$
\end{lem}
\begin{proof}
As $\forall f \in X_1$ $(m_k \cdot \psi_k (f) ) \in l^p$, ${\bf M}$
converges unconditionally and is well defined by Proposition
\ref{sec:besselbound1}.

For $n>0$ we have
\begin{eqnarray*}
\|\sum_{k=1}^{n}m_{k}\psi_k (f)\phi_k \|_{X_2}& \le & \sum_{k=1}^{n
} \| m_{k} \psi_k (f)\phi_k
\|_{X_2}\\
&\leq&\|m\|_{\infty}\left(\sum_{k=1}^{n}|\psi_k
(f)|^{p}\right)^{\frac{1}{p}}\sup_{\|h\|\leq 1}\left(\sum_{k=1}^{n}|\phi_k(h)|^{q}\right)^{\frac{1}{q}}\\
&\leq&\|m\|_{\infty} \cdot B_1\|f\|_{X_1}.\sup_{\|h\|=1}(B_2\|h\|_{X^*_2})\\
&=&\|m\|_{\infty} \cdot B_1 \cdot B_2\|f\|_{X_1}.
\end{eqnarray*}
So the multiplier is bounded with bound $\|m\|_{\infty}\cdot  B_1
\cdot B_2$.

\end{proof}

Using the representation ${\bf M}_{m, (\phi_k), ( \psi_k )} =
T_{\phi_k} D_m U_{\psi_k}$ gives a more direct way to prove the
above bound. Here $D_m$ is the diagonal operator on $\ell^\infty$
defined by $D_m(\xi_i)=(m_i \xi_i).$

 Using the above Lemma, we can define:
\begin{defn} Let $(\psi_k) \subseteq X_1^*$ be a  $p$-Bessel sequence for
$X_1$ and let $(\phi_k) \subseteq X_2$ be a $q$-Bessel sequence for
$X_2^*$.
 Let $m \in l^\infty$. 
  The
operator ${\bf M}_{m, ( \phi_k), (\psi_k)} : X_1 \rightarrow X_2$,
defined by
$$ {\bf M}_{m, (\phi_k), ( \psi_k )} (f)  =  \sum \limits_k m_k \psi_k
(f)\phi_k$$ is called \em {(p,q)-Bessel multiplier}\em. The sequence
$m$ is called the \em symbol \em of $\bf M$.
\end{defn}
\begin{prop}
Let $(\psi_k) \subseteq X_1^*$ be a  $p$-Bessel sequence for $X_1$
with no zero elements, let $(\phi_k) \subseteq X_2$ be a $p$-Riesz
sequence for $X_2$ and let $m \in l^\infty$. Then the mapping
$$m\rightarrow {\bf M}_{m, ( \phi_k), (\psi_k)}$$ is injective
from $l^\infty$ into $\mathcal{B}(X_1,X_2)$.
\end{prop}
\begin{proof} Suppose ${\bf M}_{m}={\bf M}_{m'}$, then $\sum \limits_k
m_k \psi_k (f)\phi_k=\sum \limits_k m'_k \psi_k (f)\phi_k$ for all
$f$. As $( \phi_k)$ is a $p$-Riesz basis for its span, $m_k \psi_k
(f)= m'_k \psi_k (f)$ for all $f,k$. For every $k$ there exists $f$
such that $\psi_k(f)\neq 0$, which implies that $m_{k}=m'_{k}$.
\end{proof}
\begin{prop}
Let $(\psi_k) \subseteq X_1^*$ be a  
$q$-Riesz basis for $X_1^*$ with bounds $A_1$ and $B_1$, let
$(\phi_k) \subseteq X_2$ be a $q$-frame for $X_2^*$ with bounds
$A_2$ and $B_2$ and let $m \in l^\infty$. Then
$$A_1 A_2\|m\|_{\infty}\leq \|{\bf M}_{m, (\phi_k), ( \psi_k )}\|_{Op}\leq B_1 B_2 \|m\|_{\infty}.$$
Particularly ${\bf M}$ is bounded if and only if $m$ is bounded.
\end{prop}
 \begin{proof} 
 Lemma \ref{sec:multopinf1}
 gives the upper
 bound.

  Proposition \ref{ri} states that
 $(\psi_k)$ has a biorthogonal sequence $(f_{i})\subseteq X_1$,
 i.e. $\psi_k(f_{i})=\delta_{k,i}$. $(f_{i})$ is also a Riesz basis with bounds $\frac{1}{B_1}, \frac{1}{A_1}$, and so
$\frac{1}{B_1}\leq\|f_{i}\|\leq\frac{1}{A_1}$ for all $i\in I$.
 For arbitrary $i\in I$, we have

\begin{eqnarray*}
\|{\bf M}\|_{op}= \sup_{f\in X_1}\frac{\|{\bf M}f\|}{\|f\|}\geq
\sup_{i\in I}\frac{\|{\bf M}f_{i}\|}{\|f_{i}\|} &=&
 \sup_{i\in I} \frac{\|\sum m_{k}\psi_k (f_{i})\phi_k
\|}{\|f_{i}\|}\\
&=& \sup_{i\in I} \frac{\|m_{i}\phi_i
\|}{\|f_{i}\|}\\
&=& \sup_{i\in I} |m_{i}|\frac{\|\phi_i \|}{\|f_{i}\|}\\
&\geq& A_1 A_2 \|m\|_{\infty}.
\end{eqnarray*}
So $ A_1 A_2\|m\|_{\infty}\leq \|{\bf M}\|_{op} .$
 \end{proof}
 The following proposition shows that under certain condition on $m$ the multiplier can be invertible\footnote{
 For a detailed study of invertible multiplier (on Hilbert spaces) see \cite{balsto09}.
 }, the inverse being the multiplier with the inverted symbol, similar to a result in \cite{xxljpa1} for Hilbert spaces.
 \begin{prop}
Let $(\psi_k) \subseteq X_1^*$ be a
$q$-Riesz basis for $X_1^*$, $(\phi_k) \subseteq X_2$ be a
$p$-Riesz basis for $ X_2^*$. Let
$m$ be semi-normalized (i.e. $0<\inf\mid m_k\mid\leq \mbox{$\sup\mid
m_k\mid$}<+\infty$). Then ${\bf M}_{m, (\phi_k), ( \psi_k )}$ is
invertible and $$({\bf M}_{m, (\phi_k), ( \psi_k )})^{-1}={\bf
M}_{(\frac{1}{m_k}), (\tilde \psi_k), ( \tilde \phi_k )}. $$
 \end{prop}
\begin{proof}
It is clear that $(\frac{1}{m_k})\in \ell^{\infty}$ and thus ${\bf
M}_{(\frac{1}{m_k}), (\tilde \psi_k), ( \tilde \phi_k )}$ is
well-defined. For $f\in X_1$
\begin{eqnarray*}
{\bf M}_{(\frac{1}{m_k}), (\tilde \psi_k), ( \tilde \phi_k )} \circ
{\bf M}_{m, ( \phi_k), (\psi_k)}f&=& {\bf M}_{(\frac{1}{m_k}),
(\tilde \psi_k), ( \tilde \phi_k )}(\sum \limits_k m_k
\psi_k (f)\phi_k)\\
&=&\sum_{i}\frac{1}{m_{i}}\tilde\phi_i(\sum_k m_k \psi_k
(f)\phi_k)\tilde\psi_i\\&=&\sum_{i}\frac{1}{m_{i}}\sum_k m_k \psi_k
(f)\tilde\phi_i(\phi_k)\tilde\psi_i\\
&=&\sum_{i}\psi_i(f)\tilde\psi_i\\
&=& f.
\end{eqnarray*}
That ${\bf M}_{m, ( \phi_k), (\psi_k)}\circ {\bf
M}_{(\frac{1}{m_k}), (\tilde \psi_k), ( \tilde \phi_k )} f=f$ for
all $f\in X_2$ can be shown in an analogous way.
 Hence, $$({\bf M}_{m, (\phi_k), ( \psi_k )})^{-1}={\bf
M}_{(\frac{1}{m_k}), (\tilde \psi_k), ( \tilde \phi_k )}.$$
\end{proof}

For Banach spaces it is well known that the limit of finite rank
operators (in the operator norm) is a compact operator (although
this is not an equivalent conditions as is the case for Hilbert
spaces). We are using this property in:
\begin{lem}\label{sec:compactmult1}
Let $(\psi_k) \subseteq X_1^*$ be a  $p$-Bessel sequence for $X_1$
with bound $B_1$, let $(\phi_k) \subseteq X_2$ be a $q$-Bessel
sequence for $X_2^*$ with bound $B_2$. If $m\in\mathbf{c}_0$ then
${\bf M}_{m, (\phi_k), ( \psi_k )}$ is compact.
\end{lem}
\begin{proof} For a given $m \in \mathbf{c}_0$, let $m^{(N)} = \left( m_0, m_1, \dots, m_{N-1}, 0, 0 , \dots \right)$.
The symbol $m$ is converging to zero, so for all $\epsilon > 0$
there is a $N_\epsilon$ such that $\norm{\infty}{m - m^{(N)}} \le
\epsilon$ for all $N \ge N_\epsilon$. As $m \in l^\infty$ by Lemma
\ref{sec:multopinf1} we have for all $N \ge N_\epsilon$
\begin{eqnarray*}
 \norm{Op}{ {\bf M}_{m, (\phi_k), ( \psi_k )} - {\bf M}_{m^{(N)}, (\phi_k), ( \psi_k )}} &=& \norm{Op}{{\bf M}_{\left(m - m^{(N)}\right), (\phi_k), ( \psi_k )}} \\
  &\le& \norm{\infty}{m - m^{(N)}} \cdot B_2 \cdot B_1 \\
  &\le& \epsilon \cdot B_2 \cdot B_1.
\end{eqnarray*}
Therefore ${\bf M}_{m^{(N)}, (\phi_k), ( \psi_k )}$ is converging to
${\bf M}_{{m}, (\phi_k), ( \psi_k )}$ in the operator norm. As ${\bf
M}_{m^{(N)}, (\phi_k), ( \psi_k )}$ is clearly a finite rank
operator, we have shown the result.
\end{proof}

For two normed spaces $X$ and $Y$, $\mathcal{B}(X,Y)$ denotes the
set of all linear bounded operators from $X$ to $Y$.
 Let $W$, $X$, $Y$ and $Z$ be normed spaces. For elements $y\in Y$
 and $\omega\in X^{*}$ define an operator $y\otimes\omega\in
 \mathcal{B}(X,Y)$ by $$(y\otimes\omega)(z)=\omega(z)y \quad\textrm{for all}\quad
 z \in X.$$
\par For arbitrary $S\in \mathcal{B}(W,X)$, $T\in
\mathcal{B}(Y,Z)$,
 $y\in Y$, $\omega\in X^{*}$, $z\in Z$ and $\tau\in Y^{*}$ these
 operators satisfy
 $$(z\otimes\tau)(y\otimes\omega)=\tau(y) \cdot z\otimes\omega$$
 $$T(y\otimes\omega)=T(y)\otimes\omega$$
 $$(y\otimes\omega)S=y\otimes S^{*}(\omega)$$
 $$(y\otimes\omega)^{*}=\omega\otimes\kappa(y)$$
 $$ \| y\otimes\omega \|_{Op} = \| y\|_Y \| \omega \|_{X^*} $$
 where $\kappa:Y\hookrightarrow Y^{**}$ is the canonical injection
 defined by $$\kappa(y)(\eta)=\eta(y)\quad\textrm{for all}\quad
 y \in Y, \eta\in Y^{*}.$$
The above notations are borrowed from \cite{palmer}. By using the
above notations, we can write the $(p,q)$-Bessel multiplier in the
form
$${\bf M}_{m, (\phi_k), ( \psi_k )}
=\sum_{k}m_{k}\phi_{k}\otimes\psi_{k}.$$ It is easy to see that
$${\bf M}_{m, (\phi_k), ( \psi_k )}^{*}=\sum_{k}\overline{m_{k}}\psi_{k}\otimes\kappa(\phi_{k})={\bf M}_{\overline{m}, (\psi_k), ( \kappa(\phi_{k} ))}.$$

Putting the above results together, we obtain the following theorem
which is a generalization of one of the results in \cite{xxlmult1}
for Banach spaces.
\begin{thm} \label{sec:bessmulprop1} \it Let ${\bf M} = \MM_{m,(\phi_k),(\psi_k)}$
 be a $(p,q)$-Bessel multiplier for the $p$-Bessel sequence $(\psi_k) \subseteq X_1^*$,
 the $q$-Bessel sequence $(\phi_k) \subseteq X_2$ with  bounds $B_1$ and $B_2$. 
Then, the following hold.
\begin{enumerate}
\item If $m \in l^\infty$,
${\bf M}$ is a well defined bounded operator with
 $$\norm{Op}{\bf M} \le  B_2 B_1 \cdot
\norm{\infty}{m} .$$ Furthermore, the sum $\sum \limits_k m_k \psi_k
(f)  \phi_k$ converges unconditionally for all $f \in X_1$.
\item ${\bf M}_{m, (\phi_k), ( \psi_k )}^{*}=\sum_{k}\overline{m_{k}}\psi_{k}\otimes\kappa(\phi_{k})={\bf M}_{\overline{m}, (\psi_k), ( \kappa(\phi_{k} ))}.$

\item If $m \in \mathbf{c}_0$,
 $\bf M$ is a compact operator.
\end{enumerate}
\end{thm}

\section{Nuclear operators in Banach spaces} \label{sec:nucl0}
The theory of trace-class operators in Hilbert spaces was created in
1936 by J. Murray and J. Von Neumann. In the earlier Fifties,
Alexander Grothendieck \cite{Gro55} and A. F. Ruston \cite
{Rus1,Rus2} independently extended this concept to operators acting
in Banach spaces. Trace-class operators on Banach spaces are called
nuclear operators. This idea is generalized in \cite{Pitch}:

Let $0<p\leq\infty$. A family $\textbf{x}=(x_{i})_{i\in I}\subseteq
X$, where $x_{i}\in X$ for $i\in I$, is called \textit{weakly
$p$-summable} if $(x^{*}(x_{i}))\in \ell^{p}(I)$ whenever $x^{*}\in
X^{*}$.
\par We put $$w_{p}(x_{i}):=\sup \{\|{x^{*}(x_{i})}\|_{p}:\|x^{*}\|\leq 1
\}.$$ The class of all weakly $p$-summable sequences on $X$ is
denoted by \textbf{$W_{p}$}(X). Clearly $w_{p}(x_{i}) < \infty$ ( by
Banach-Steinhaus Theorem ).
\par From the above notations, it is clear that if $(g_{i})_{i\in I}\subset
X^{*}$ is a $p$-Bessel sequence for $X$ then
 $(g_{i})\in W_{p}(X^{*})$\footnote{As mentioned in the introduction we only consider reflexive Banach spaces}.
\begin{defn}\cite{Pitch}
Let $0<r\leq\infty$, $1\leq p_1,q_1\leq\infty$, and
$1+\frac{1}{r}\geq\frac{1}{p_1}+\frac{1}{q_1}.$ An operator
$S\in\mathcal{B}(X,Y)$ is called $(r,p_1,q_1)$-nuclear if
$$S=\sum_{i=1}^{\infty}\sigma_{i}x_{i}^{*} \otimes y_i$$ with $(\sigma_{i})\in
\ell^{r}$, $(x_{i}^{*})\in W_{q'}(X^{*})$, and $(y_{i})\in
W_{p'}(Y)$ where
$\frac{1}{p_1}+\frac{1}{p'}=\frac{1}{q_1}+\frac{1}{q'}=1$. In the
case $r=\infty$ let us suppose that $(\sigma_{i})\in\mathbf{c}_{0}$.
We put
$$N_{(r,p_1,q_1)}(S):=\inf \left\{
\|(\sigma_{i})\|_{r} \cdot w_{q'}(x^{*}_{i}) \cdot w_{p'}(y_{i})
\right\},$$ where the infimum is taken over all so-called
$(r,p_1,q_1)$-nuclear representations described above.
\end{defn}
\begin{thm} \cite{Pitch} \label{rpq}
An operator $S\in \mathcal{B}(X,Y)$ is $(r,p_1,q_1)$-nuclear if and
only if there exist operators $F$, $D$ and $E$ with $S = F D E$,
such that $D_\sigma\in \mathcal{B}(\ell^{q'},\ell^{p_1})$ is a
diagonal operator of the form
$D_\sigma(\xi_{i})=(\sigma_{i}\xi_{i})$ with
$(\sigma_{i})\in\ell^{r}$ if $0<r<\infty $ and
$(\sigma_{i})\in\textbf{c}_{0}$ if $r=\infty$ . Furthermore, $E\in
\mathcal{B}(X,\ell^{q'})$ and $F\in \mathcal{B}(\ell^{p_1},Y)$. In
this case,
$$N_{(r,p_1,q_1)}(S): = \inf \|E\|\|(\sigma_{i})\|_{r}\|F\|,$$ where
the infimum is taken over all possible factorizations.
\end{thm}

 From  Theorem \ref{rpq} and the above notations, we can easily conclude the next result for multipliers using $${\bf M}_{m, (\phi_k), ( \psi_k )} = T_{\phi_k} D_m
U_{\psi_k}$$ as a decomposition of ${\bf M}$.
\begin{cor} \label{multrpg}
Let $(\psi_k) \subseteq X_1^*$ be a  $p$-Bessel sequence for $X_1$
with bound $B_1$, let $(\phi_k) \subseteq X_2$ be a $q$-Bessel
sequence for
$X_2^*$ with bound $B_2$. 
 Let $r>0$ and $m \in \ell^r$. Then
${\bf M}_{m, (\phi_k), ( \psi_k )}$ is a $(r,p,q)$-nuclear operator
with
$$N_{(r,p,q)}({\bf M}) \le B_1 B_2 \| m \|_r .$$
\end{cor}

\section{Changing the ingredients} \label{sec:changing0}

Results from \cite{xxlmult1} can be generalized to 
$p$-Bessel sequences:
\begin{thm} \label{sec:frammulprop1}
Let $\MM = \MM_{m,(\phi_k),(\psi_k)}$
 be a $(p,q)$-Bessel multiplier for the $p$-Bessel sequences $(\psi_k) \subseteq X_1^*$,
 the $q$-Bessel sequence $(\phi_k) \subseteq X_2$ with  bounds $B_1$ and
 $B_2$. Let $p_1, q_1 \ge 1$ be such that $\frac{1}{p_1}+ \frac{1}{q_1} = 1$ allowing $p_1, q_1 = \infty$.
  Then the operator $\MM$ depends continuously on $m$, $(\psi_i)$ and $(\phi_i)$, in
   the following sense: Let $(\psi_i^{(l)}) \subseteq X_1^*$ and $(\phi_i^{(l)}) \subseteq X_2$ be Bessel sequences
      \footnote{Please note that for a convergence of $p$-Bessel sequences in an $l^{p}$-sense we would get the Bessel property by Corollary \ref{sec:framesimi1} for big enough $l$. }
   indexed by $l \in I$.
\begin{enumerate}
\item Let $m^{(l)} \rightarrow m$ in $l^{p_1}$. Then
$ \norm{Op}{M_{m^{(l)},(\psi_i),(\phi_i)} - M_{m,(\psi_i),(\phi_i)}}
\rightarrow 0 $.
\item Let $m \in l^{p_1}$ and let the sequences $(\psi_i^{(l)})$ converge to $(\psi_i)$ in an $l^{q_1}$-sense.
Then for $l \rightarrow \infty$
$$\norm{Op}{M_{m,(\psi_i^{(l)}),(\phi_i)} - M_{m,(\psi_i),(\phi_i)}}
\rightarrow 0.$$
\item Let $m \in l^{p_1}$ and let the sequences $(\phi_i^{(l)})$ converge to $(\phi_i)$ in an $l^{q_1}$-sense.
Then for $l \rightarrow \infty$
$$
\norm{Op}{M_{m,(\psi_i),(\phi_i)} - M_{m,(\psi_i),(\phi_i^{(l)})}}
\rightarrow 0.$$
\item Let $m^{(l)} \rightarrow m$ in $l^{p_1}$ and let the sequences $(\psi_i^{(l)})$ respectively $(\phi_i^{(l)})$
converge to $(\psi_i)$ respectively $(\phi_i)$ in an
$l^{q_1}$-sense. Then for $l \rightarrow \infty$ $
\norm{Op}{M_{m^{(l)},(\psi_i^{(l)}),(\phi_i^{(l)})} -
M_{m,\psi_i,\phi_i}} \rightarrow 0$.
\end{enumerate}
\end{thm}

\begin{proof}
\begin{enumerate}

\item
By Theorem \ref{sec:bessmulprop1}
\begin{eqnarray*}
\norm{Op}{\MM_{m^{(l)},(\psi_k),(\phi_k)} -
\MM_{m,(\psi_k),(\phi_k)}}&=&\norm{Op}{\MM_{\left( m^{(l)} -
m\right),(\psi_k),(\phi_k)}}\\
&\leq& \norm{\infty}{m^{(l)}-m}
\sqrt{B_1 B_2}\\
&\leq& \| m^{(l)} - m\|_{p_1}\sqrt{B_1 B_2}\\
&\leq& \epsilon \sqrt{B_1 B_2}.
\end{eqnarray*}
for $l > N_{\epsilon}$.

\item 
For $l > N_\epsilon$
\begin{eqnarray*} \norm{Op}{\sum m_k \psi_k^{(l)} \otimes_i \phi_k - \sum m_k \psi_k \otimes_i \phi_k}&=& \norm{Op}
{\sum m_k \left( \psi_k^{(l)}- \psi_k \right) \otimes_i \phi_k} \\
&\leq&  \sum \limits_k \left| m_k \right| \norm{X_1^*}{\psi_k^{(l)}-
\psi_k} \sqrt{B_2} \\ &\leq& \sqrt{B_2} \norm{p_1}{m} \left( \sum
\norm{X_1^*}{\psi_k^{(l)}- \psi_k}^{q_1}\right)^{1/q_1}\\&\leq&
\sqrt{B_2} \norm{p_1}{m} \varepsilon.
\end{eqnarray*}

\item Use corresponding arguments as in (2).

\item
\begin{eqnarray*} \norm{}{M_{m^{(l)},(\psi_k^{(l)}),(\phi_k^{(l)})} - M_{m,(\psi_k),(\phi_k)}} &\leq& \norm{}{M_{m^{(l)},
(\psi_k^{(l)}),(\phi_k^{(l)})} -
M_{m,(\psi_k^{(l)}),(\phi_k^{(l)})}}  \\
& +& \norm{}{M_{m,(\psi_k^{(l)}),(\phi_k^{(l)})} -
M_{m,(\psi_k),(\phi_k^{(l)})}}\\& +&
\norm{}{M_{m,(\psi_k),(\phi_k^{(l)})}
 - M_{m,(\psi_k),(\phi_k)}}
\\& \le& \varepsilon \sqrt{B_1 B_2} + \norm{p_1}{m} \varepsilon \sqrt{ B_2}
+ \norm{}{m} \sqrt{B_1} \varepsilon \\&=& \varepsilon \cdot \left(
\sqrt{ B_1 B_2} + \norm{p_1}{m} \left( \sqrt{B_2} + \sqrt{B_1}
\right) \right)
\end{eqnarray*} for $l$ bigger than the maximum
$N$ needed for the convergence conditions.
\end{enumerate}
\end{proof}

\section{Outlook and Perspectives}

We have shown that the concept of multipliers can be extended to
$p$-frames in Banach spaces. It can also be done for other settings.
For $g$-frames a paper was already accepted \cite{rahgenmul10}. In
the future we will consider to extend this notion to other setting,
for example for Frechet frames, matrix valued frames, $pg$-frames
\cite{pg} or continuous frames.

In particular the last notion can be interesting also for
application as in this setting the question, how continuous and
discrete frame multipliers can be related, is of relevance. This
would be an interesting result for the link of STFT and Gabor
multipliers. Such a connection is particular interesting in relating
a physical model, which normally is continuous, using multipliers to
the implemented algorithm, which is discrete and finite-dimensional.

For the future work the relation of $(p,q,r)$-nuclear operators with
Gelfand triple may be investigated.

 For applications wavelet, Gabor
and frames of translates are very important classes of frames. Most
of these systems can be described as localized frames
\cite{forngroech1}. For multipliers of localized frames, which are
currently investigated, the results of this paper are directly
applicable and so can become more relevant for signal processing
algorithms.
We further hope that the results in this paper can be directly
useful for applications in signal processing, as both Banach space
methods and multipliers become more and more important for
applications, as mentioned in the introduction.

\section*{Acknowledgment}
Some of the results in this paper were obtained during the first
author's visit at the Acoustics Research Institute, Austrian Academy
of Sciences, Austria. He thanks this institute for their
hospitality.

This work was partly supported by the WWTF project MULAC ('Frame
Multipliers: Theory and Application in Acoustics; MA07-025).

The authors would like to thank Diana Stoeva for her discussions and
comments. Also, the authors would like to thank referees for their
suggestions.

\end{document}